\documentclass[12pt]{amsart}
\usepackage{amsmath}
\usepackage{amsthm}
\usepackage{amssymb}
\usepackage{verbatim}
\usepackage{subfigure}
\usepackage{verbatim}
\usepackage{url}
\numberwithin{equation}{section}
\usepackage{textcomp}
\usepackage{stmaryrd}
\usepackage{xy}
\usepackage{color}
\usepackage[pdftex]{graphicx}
\usepackage{epstopdf}
\usepackage{mathrsfs}
\definecolor{Mygrey}{gray}{0.75}

\def\displayandname#1{\rlap{$\displaystyle\csname #1\endcsname$}%
                      \qquad \texttt{\char92 #1}}

\makeatletter
\def\url@leostyle{%
  \@ifundefined{selectfont}{\def\UrlFont{\sf}}{\def\UrlFont{\small\ttfamily}}}
\makeatother
\newtheorem{thm}{Theorem}[section]
\newtheorem{pro}[thm]{Proposition}
\newtheorem{lem}[thm]{Lemma}
\newtheorem{cla}[thm]{Claim}
\newtheorem{con}[thm]{Conjecture}
\newtheorem{cor}[thm]{Corollary}

\theoremstyle{definition}
\newtheorem{df}[thm]{Definition}

\theoremstyle{remark}
\newtheorem{rem}[thm]{Remark}
\newtheorem{exa}[thm]{Example}

\begin{document}

\bibliographystyle{acm}

\title{On pseudo-inverses of matrices and their characteristic polynomials in supertropical algebra}
  \author[Adi Niv]{ Adi Niv$^\dagger$ }

  \thanks{$^\dagger$ Department of Mathematics, Bar-Ilan University, Ramat Gan 52900, Israel.\newline
Email:  {\tt adi.niv@live.biu.ac.il}} 
\thanks{This paper is part of the author's PhD thesis, which was written at Bar-Ilan University under the supervision of Prof. L.\ H.\ Rowen.} 
\thanks{The author wishes to thank Dr. Sergei Sergeev from the University of Birmingham and an anonymous referee for their invaluable comments which were very helpful when preparing this paper.}
\thanks{This research was supported by the Israel Science Foundation (grant no. 1207/12).}

\begin{abstract} The only invertible matrices in tropical algebra are diagonal matrices, permutation matrices and their products. However, the  pseudo-inverse~$A^{\nabla}$, defined as $\frac{1}{det(A)} adj(A)$, with $det(A)$ being the tropical permanent (also called the tropical determinant) of a matrix $A$,  inherits some classical algebraic properties and has some surprising new ones. Defining $B$ and $B'$ to be tropically similar if~$B'=A^{\nabla} B A$, we examine the characteristic (max-)polynomials of   tropically similar matrices as well as those of pseudo-inverses. Other miscellaneous results include a new proof of the identity for~$det(AB)$ and a connection to stabilization of the powers of definite matrices.

\vskip 0.25 truecm

\noindent \textit{Keywords: Tropical and supertropical linear algebra; characteristic polynomial; eigenvalues; Kleene star; permanent; definite matrices; pseudo-inverse.}
\vskip 0.25 truecm

\noindent \textit{AMSC: 15A09 (Primary), 15A15, 15A18, 15A80, 15B33} 	
\end{abstract}

\maketitle

%First page headline in LaTeX for S\'eminaire Lotharingien de Combinatoire
%--first part
\thispagestyle{myheadings}
\font\rms=cmr8
\font\its=cmti8
\font\bfs=cmbx8

\markright{}
\def\thepage{}

 \section{Introduction}

The \textbf{tropical max-plus semifield} is an ordered group $\mathcal{G}$ (usually the set of real numbers~$\mathbb{R}$ or the set of rational numbers~$\mathbb{Q}$),  together with~$-\infty$, denoted as~$\mathbb{T}=\mathcal{G}\bigcup\{-\infty\}$, equipped with the operations $a\varoplus b=max\{a,b\}$ and $a\varodot b=a+b$, denoted as~$a+b$ and~$ab$ respectively (see ~\cite{MPA} and ~\cite{TAG}). The unit element $1_{\mathbb{T}}$ is actually the element~$0\in \mathbb{Q}$. This arithmetic enables one to simplify non-linear questions by asking them in a (pseudo)-linear  setting (see~\cite{PA}), which can be applied to discrete mathematics, optimization, algebraic geometry and more, as has been well reviewed in ~\cite{LOP}, ~\cite{MAA}, ~\cite{S&A}, ~\cite{MPW}, ~\cite{IM}  and ~\cite{TGA}. In this max-plus language, we may also use notions of linear algebra to interpret combinatorial problems, such as eigenvectors being used to solve the Longest-Distance problem (see ~\cite{MA}).

The intention of this paper is to use an analogous concept of the inverse matrix by passing to a wider structure called the \textbf{supertropical semiring}, equipped with the ghost ideal~$G=\mathcal{G}^{\nu}$, as established and studied by Izhakian and Rowen in ~\cite{STA} and ~\cite{STLA}. The use of the term \textgravedbl pseudo\textacutedbl in this paper is the same as \textgravedbl quasi\textacutedbl throughout the work of Izhakian and Rowen. 

We denote the \textgravedbl standard\textacutedbl\  supertropical semiring  as $$R=T\bigcup G\bigcup \{-\infty \},$$ where~$T=\mathcal{G}$, which contains the so called tangible elements of the structure and where~$\forall a\in T$ we have~$a^{\nu}\in G$ are the ghost elements of the structure, as defined in~\cite{STA}. So~$G$ inherits the order of~$\mathcal{G}$. This enables us to distinguish between a maximal element~$a$ that is  attained only once in a sum, i.e.~$a\in T$ which is invertible, and a maximum that is being attained at least twice, i.e.~$a+a=a^{\nu}\in G$, which  is not invertible. Note that $\nu$ projects the standard supertropical semiring onto~$G$, which can be identified with the usual tropical structure.

In this new supertropical sense, we use the following order relation to describe two elements that are equal up to some ghost supplement:

\begin{df}
Let $a,b$ be any two elements in $R$. We say that $a$ \textbf{ghost surpasses}~$b$, denoted $a\models_{gs} b$, if~$a=b+ghost$, i.e. $a=b\ $ or $a\in G$ with $a^{\nu}\geq b^{\nu}$.  We say $a$ is~$\nu$-equivalent to~$b$, denoted by~$a\cong_{\nu}b$, if~$a^{\nu}=b^{\nu}$.That is, when we use $\nu$ to project from $R$ to $G$, identified with the tropical structure, $\nu$-equivalence becomes equality. 

For matrices $A=(a_{ij}),B=(b_{ij})\in M_{n\times m}(R)$ (and in particular for vectors) $A\models_{gs} B$ means $a_{ij}\models_{gs} b_{ij}\quad \forall i=1,...,n$ and $j=1,...,m$.

For polynomials~$$f(x)=\sum_{i=1}^n a_ix^i,\ g(x)=\sum_{i=1}^n b_ix^i\in R[x],$$ we say that~$f(x)\models_{gs} g(x)$, also denoted as $f \models_{gs}g$, when~$a_{i}\models_{gs}b_{i}\ \forall i$.
\end{df}
\vskip 0.25 truecm

%First page headline in AmS-LaTeX for S\'eminaire Lotharingien de Combinatoire
%--restoring the headers and pagenumbering
\pagenumbering{arabic}
\addtocounter{page}{1}
\markboth{\SMALL ADI NIV }{\SMALL On pseudo-inverses of matrices and their characteristic polynomials in supertropical algebra}

    \textbf{Important properties of $\models_{gs}$:}

\noindent{1. $\models_{gs}$ is a partial order relation.}

See ~\cite[Lemma 1.5]{STMA2}.

\vskip 0.25 truecm

\noindent{2. If $a\models_{gs} b$ then $ac\models_{gs} bc$.}

\vskip 0.25 truecm

Considering this relation, we use the classical notion~$\frac{1}{det(A)}adj(A)$ (where $det(A)$ is the permanent and then $adj(A)$ is defined as usual) to formulate results in the supertropical setting, which are inaccessible in the usual tropical setting. This notion was introduced and studied in tropical algebra  by Yoeli ~\cite{YL} and Cuninghame-Green ~\cite{MNMX}, and was further investigated by Gaubert ~\cite{G.Ths}, Reutenauer and Straubing ~\cite{MCS} and Sergeev ~\cite{KS}.
We obtain tropical theorems by considering the tangible elements. By Izhakian's results in ~\cite{TA}, this notion satisfies that $\frac{1}{det(A)}adj(A)\cdot A$ is equal to the $Id$ matrix on the diagonal, and $\models_{gs}$ the $Id$ matrix off the diagonal.

In section 3 we discuss type of matrices with $0$ on the diagonal and $0$ determinant, defined as \textbf{definite matrices}. In section 3.1 , we establish that for the set $$\mathcal{N}=\{A^\nabla: A\in M_n(R)\},$$ the operation $\nabla$ is of order 2 (see Theorem 3.5), and in section 3.2 we will show that for a definite matrix $A$ of order $n$ $$A^\nabla\cong_{\nu}A^{\nabla\nabla}\cong_{\nu}A^*\cong_{\nu}A^{k},\ \forall k\geq n-1$$ (see Theorem 3.6 and Proposition 3.7). These results are extended to the supertropical setting from the  results obtained over the tropical structure in ~\cite{YL} and ~\cite{KS}, regarding these closure operations. 

In section 4, we use the factorizability of matrices in $\mathcal{N}$ to give an alternative proof, analogous to the proof in classical linear algebra, of the property $$det(AB)\models_{gs}det(A)det(B)$$ stated in Theorem 2.10. 

In sections 5 we prove that

$$f_{A^\nabla BA}\models_{gs}f_B,\text{ (see Theorem 5.4)}.$$  

In section 6, our considerations lead us to Conjecture 6.2:
$$|A|f_{A^\nabla}(x)\models_{gs}x^nf_A(x^{-1}),$$

\noindent  where $f_M(x)=det(M+xI)$ denotes the characteristic polynomial of a square matrix~$M$.  We prove 6.2 for  the determinant and trace coefficients and for every $2\times 2$ or $3\times 3$ and triangular matrix.

A consequence would be
$$f_{A^{\nabla\nabla}}=f_A,\text{ whenever } f_{A^{\nabla\nabla}} \text{ is tangible.}$$

These properties provide a foundation for further research in representation theory and eigenspace decomposition.

\section{Preliminaries}

\subsection{Matrices}

$ $

The work in ~\cite{STMA}, ~\cite{STMA2} and ~\cite{STMA3} shows that even though the semiring of matrices over the supertropical semiring lacks negation, it satisfies many of the classical matrix theory properties when using the ghost ideal $G$. Following ~\cite{MA}  and ~\cite{STMA} for the tropical and supertropical notations, we give some basic definitions for this theory. One may also find in ~\cite{MA} 
further combinatorial motivations for the objects discussed.
\vskip 0.25 truecm

\begin{df}
A \textbf{permutation track, }of the permutation $\pi\in S_n$, is the sequence $$a_{1,\pi(1)}a_{2,\pi(2)}\cdots a_{n,\pi(n)}$$ of $n$ entries of the matrix $A=(a_{i,j})\in M_n(R)$.
\end{df}

\begin{df}
We define the \textbf{tropical determinant} of a  matrix $A=(a_{i,j})$ to be 

    $$det(A)=\sum_{\sigma \in S_n}a_{1,\sigma(1)}\cdots a_{n,\sigma(n)}.$$  

In the special case where $a_{i,j}\in R,\ \forall i,j$, we refer to any permutation track yielding the highest value in this sum as \textbf{a dominant permutation track}.
\end{df}

The tropical determinant is actually the same as \textit{max permanent} of ~\cite{MA}-~\cite{MNMX}, and the above definition is  a reformulation of the optimal assignment problem.

\begin{df}
We define a matrix~$A\in M_n(\mathbb{T})$ or~$A\in M_n(R)$  to be  \textbf{tropically singular} if there exist at least two different dominant permutation tracks. Otherwise the matrix is \textbf{tropically non-singular}.\end{df}

Consequently  a matrix~$A\in M_n(R)$ is supertropically singular if~$det(A) \in G\bigcup\{-\infty\}$ and supertropically non-singular if~$det(A) \in T$. A matrix $A$ is \textbf{strictly singular} if~$det(A)=-\infty$.  
\vskip 0.25 truecm

Notice that over the tropical semifield we cannot determine if the matrix is tropically non-singular from the value of its determinant, which is always invertible over $\mathbb{T}\setminus\{-\infty\}$. Over the supertropical semiring however, a supertropically non-singular matrix  has an invertible determinant, while a supertropically singular matrix  has a non-invertible determinant. 

As the  definitions of singularity are identical over the tropical and supertropical structures, we will only indicate \textgravedbl non-singular\textacutedbl\  or \textgravedbl singular\textacutedbl\  and \textgravedbl  over  $\mathbb{T}$\textacutedbl \ or \textgravedbl over ~$R$\textacutedbl (which will effect the value and invertibility of the determinant).% When talking about non-singular matrices over the field $K$ we use \textgravedbl invertible\textacutedbl.

\begin{df}
Let $\mathbb{T}^n$ be the free module (see ~\cite{STMA2}) of rank $n$ over the tropical  semifield, and $R^n$ be the free module of rank $n$ over the supertropical  semiring. We define the \textbf{standard base} of $\mathbb{T}^n$, and therefore of $R^n$, to be $e_1,...,e_n$, where 

$e_i=
\begin{cases}
1_{\mathbb{T}}=1_R,\ \text{in the}\ i^{th}\ \text{coordinate}\\
0_{\mathbb{T}}=0_R,\ \text{otherwise} 
\end{cases}$.
\end{df}

\begin{df} The tropical \textbf{identity matrix} in the tropical matrix semiring is 

\noindent the~$n\times n$ matrix with the standard base for its columns. We denote this matrix as $$I_{\mathbb{T}}=I_R=I.$$
\end{df}

\begin{df}
A matrix $A\in M_n(R)$ is \textbf{ invertible} if there exists a matrix~$B\in M_n(R)$ such that $$AB=BA=I.$$
\end{df}

\begin{df} A square matrix $P_{\pi}=(a_{i,j})$ is defined to be a \textbf{permutation matrix} if there exists $\pi\in S_n$ such that 
$a_{i,j}=
\begin{cases}
0_R, j\ne \pi(i)\\
 1_R, j=\pi(i)
\end{cases}.$
\end{df}

\begin{rem}
 A tropical matrix $A$ is invertible if and only if it is  a product of a permutation matrix $P_{\pi}$ and a diagonal matrix~$D$ with an invertible determinant. These type of products are defined as \textbf{generalized permutation matrices}.
\end{rem}

\begin{df} Following the notation in ~\cite{FTM}, we define  three types of tropical elementary matrices, corresponding to the three elementary  matrix operations,  obtained by applying
one such operation to the identity matrix. 

A \textbf{transposition matrix} is obtained from the  identity matrix by switching two rows (resp. columns). Multiplying a matrix~$A$ to the right of such a matrix (resp. to the left) will switch the corresponding rows (resp. columns) in $A$.

A \textbf{diagonal multiplier} is obtained from the  identity matrix where one row (resp. column) has been multiplied by an invertible scalar. Multiplying a matrix~$A$ to the right of such a matrix (resp. to the left) will multiply the corresponding row (resp. columns) in $A$ by the same scalar.

A \textbf{Gaussian matrix} is defined to differ from the identity matrix by having a non-zero entry $r$, in a non-diagonal position $(i,j)$. Multiplying a matrix $A$ to the right of such a matrix (resp. to the left) will add row $j$ (resp. column), multiplied by $r$, to row (resp. column) $i$.
\end{df}

By Remark 2.8, a product of transposition matrices, which is a permutation matrix, is invertible.  A product of diagonal multipliers, which is a diagonal matrix, is invertible. Thus a product of transposition matrices and diagonal multipliers, which is a generalized permutation matrix, is invertible. Gaussian matrices   however, are not invertible, and therefore a product including a Gaussian matrix  is not invertible.

\begin{thm} \textbf{(The rule of determinants)}
For $n\times n$ matrices A,B over the supertropical  semiring R, we have $$det(AB) \models_{gs} det(A)det(B).$$
\end{thm}

\noindent{\textbf{Proof.}} See Theorem 3.5 in ~\cite[\S 3]{STMA}. 
\vskip 0.5 truecm

This theorem also follows from  ~\cite[Proposition 2.1.7]{G.Ths}. Although S. Gaubert's proof is done in the symmetrized tropical semiring, it carries over to the supertropical case.

In this context, we note that the two transfer principles in ~\cite[Theorem 3.3 and Theorem~3.4 ]{LDTS}  allow one to obtain such results automatically in a wider class of semirings, including the supertropical semiring.

The fact that the determinant of a non-singular matrix $A$ over $R$ is tangible means that the matrix has one dominant permutation track. By using  a permutation matrix we can relocate the corresponding permutation to the diagonal, and by using a diagonal matrix we can change the diagonal entries to~$1_R$, obtaining a non-singular matrix whose dominant Id-permutation track equals~$1_R$. That is, $A=P\bar{A}$ where $P$ is an invertible matrix (See Remark 2.8) such that $det(P)=det(A)$ and~$\bar{A}$ is a definite matrix. 

\begin{df} $\bar{A}$  is called the \textbf{ definite form of $A$},  $P$ the \textbf{conductor  of $A$} and we say that $P$ \textit{conducts} the dominant permutation track in $A$. 

The  process of bringing  a non-singular matrix $A$ to its definite form, can be obtained by using transposition matrices and diagonal multipliers either on the columns or on the rows of $A$. Respectively, we denote these conducting matrices as the \textbf{right conductor} and the \textbf{left conductor} of $A$, which are  invertible. The \textbf{right} (resp. left) \textbf{definite form} corresponds to the right (resp. left) conductor.\end{df}
\vskip 0.25 truecm

Definite matrices are not the same as normal matrices, defined in ~\cite{TLI} and ~\cite{MA} to have non-positive off diagonal entries. Over $\mathbb{T}$ the definite form is obtained for not strictly singular matrices, by conducting one of the dominant permutation tracks.
\vskip 0.25 truecm

Looking at  elementary matrices as the \textgravedbl atoms\textacutedbl\   of matrices, we present the following definitions: 

\begin{df} 
We define a matrix to be \textbf{elementarily factorizable} if it can be factored into a product of tropical elementary matrices. 
\end{df}

By Remark 2.8, an invertible matrix  is always elementarily factorizable, while  a non-invertible matrix is elementarily factorizable if and only if its definite form  is factorizable.
\vskip 0.25 truecm

\noindent \textit{Remark}. Writing a permutation $\pi$  as a product of disjoint cycles $\sigma_1,...,\sigma_t$,  the permutation track $a_{1,\pi(1)}\cdots a_{n,\pi(n)}$ can be decomposed into the \textbf{cycle tracks} $C_1,...,C_t$, where~$C_i$ is the cycle track corresponds to the cycle $\sigma_i$.
\vskip 0.25 truecm

We introduce a standard combinatorial property of definite matrices. Since this property will be in use in Section 5, we provide some details regarding the technique of its use.

\begin{lem} For a definite matrix $A=(a_{i,j})\in M_n(R)$, any sequence \begin{equation}a_{i_1,i_2}a_{i_2,i_3}\cdots a_{i_k,i_1},\text{ where }\ i_j\in \{1,...,n\}\ \text{ (not necessarily distinct)}\end{equation}  either describes a permutation track, or is dominated by a subsequence  which describes a permutation track.\end{lem}

\begin{proof} 

\noindent As proved in ~\cite[Theorem 4.4]{MA}, due to the determinant obtained by the  $0$-entries on the diagonal of a definite matrix, for every cycle track $C$ we have that $$C=C\cdot C_{Id}\leq P_{Id},$$ where $P_{Id}$ is the Id-permutation track and $C_{Id}$ is the Id-cycle track  extending $C$ to a permutation track. 

We use a standard fact from graph theory, that if a path $(i_1,i_2,i_3,\cdots ,i_k,i_1)$ has repeating nodes, then it can be recursively separated into (not necessarily disjoint) cycles,  starting and ending at the points of repetition. Thus, (2.1) can be  decomposed as the product of (not necessarily disjoint) cycle tracks satisfying $$\prod_{j=1}^l C_j\leq C_i=C_i\cdot C_{Id}\ \ \forall i,$$ where $C_j$ are cycle tracks and $C_i\cdot C_{Id}$ is the extension of $C_i$ to a permutation track.
\vskip 0.25 truecm

\end{proof}

\noindent\textbf{Notation.} For an element $a\in R$, we denote as ${a^\nu}$   the element $b\in G$ s.t.~$a\cong_\nu b$, and as  ${\hat{a}}$  the element  $b\in T$ s.t.~$a\cong_\nu b$. 
\vskip 0.15 truecm

\noindent For a matrix (and in particular for a vector) $A=(a_{i,j})$ we write  $${A^\nu}=(a_{i,j}^\nu)\ \text{{and}}\ {\widehat{A}}=(\widehat{a_{i,j}}).$$ 
\vskip 0.15 truecm

\noindent For a polynomial $f(x)=\Sigma a_ix^i$ we write  $${f^\nu(x)}=\Sigma a_i^\nu x^i\ \text{{and}}\ {\widehat{f(x)}}=\Sigma \hat{a_i}x^i.$$

\begin{df} A \textbf{pseudo-zero} matrix $Z_G$ is a matrix equal to $0_R$ on the diagonal, and whose off-diagonal entries are ghosts or $0_R$.  

A \textbf{pseudo-identity} matrix $I_G$ is a nonsingular, multiplicatively idempotent matrix equal to $I + Z_G$, where $Z_G$ is a pseudo-zero matrix. Thus, for every matrix $A$ and a pseudo-identity $I_G$ we have $I_GA\models_{gs}A$.

 A \textbf{ghost pseudo-identity} matrix is a singular, multiplicatively idempotent matrix equal to $I^\nu + Z_G$.
\end{df}

\begin{df} The ${r,c}$-\textbf{minor} $A_{r,c}$ of a matrix $A = (a_{i,j})$ is obtained by deleting  row~$r$ and column $c$ of $A$. The \textbf{adjoint matrix} $adj(A)$ of $A$ is defined as the matrix~$(a'_{i,j} )$, where
$a'_{i,j} =det(A_{j,i})$. 

The matrix $A^{\nabla}$ denotes $\frac{1}{det(A)}adj(A)$, when $det(A)$ is invertible, and $\left(\frac{1}{\widehat{det(A)}}\right)^\nu adj(A)$, when $det(A)$ is not invertible. Thus over~$R,\ A^{\nabla}$ is defined differently for singular and non-singular matrices. Over $\mathbb{T}$, however, $A^{\nabla}$ is defined the same for every not strictly singular matrix. 

\noindent\textbf{Remark.} Notice that $det(A_{j,i})$ may be obtained as the sum of all permutation tracks in $A$  passing through $a_{j,i}$, with $a_{j,i}$ removed:    $$det(A_{j,i})=\sum_{\tiny{\begin{array}{cc}\pi\in S_n:\\\pi(j)=i\end{array}}}a_{1,\pi(1)}\cdots a_{j-1,\pi(j-1)}a_{j+1,\pi(j+1)}\cdots a_{n,\pi(n)}.$$ When writing  such a permutation as the product of  disjoint cycles, $det(A_{j,i})$ can be presented as: \begin{equation}det(A_{j,i})=\sum_{\tiny{\begin{array}{cc}\pi\in S_n:\\\pi(j)=i\end{array}}}(a_{i,\pi(i)}\cdots a_{\pi^{-1}(j),j})C_{\pi},\end{equation} where $(a_{i,\pi(i)}\cdots a_{\pi^{-1}(j),j})$ is the cycle track missing $a_{j,i}$,  and $C_{\pi}$ is the product of the cycle tracks of $\pi$ that do not include $i$ and $j$.

\end{df}
\noindent\textbf{Fact.} The products $AA^{\nabla}\text{ and } A^{\nabla}A$  are pseudo-identities when $det(A)$ is invertible, and ghost pseudo-identities otherwise.

\noindent These identities can be deduced from~\cite[Proposition 2.1.2]{G.Ths}, by replacing Gaubert's~$\ominus a$ with~$a$. Then $a^\nu\models 0_R$corresponds to  $a^{\bullet}\nabla 0_R$. See ~\cite[Theorem 2.8]{STMA2} for the  proof in the supertropical setting.

\begin{df}  We say that $A^{\nabla}$ is \textbf{the pseudo-inverse} of $A$ over $R$, denoting $$I_A = AA^{\nabla}\text{ and }I'_A=A^{\nabla}A.$$\end{df}

\begin{thm} $ $

\noindent (i) $det(A\cdot adj(A))= det(A)^n$ .

\noindent (ii) $det(adj(A))= det(A)^{n-1}$ .
\end{thm}

\noindent \textbf{Proof.} ~\cite[Theorem 4.9]{STMA}.

\begin{rem} For a definite matrix $A$, we have $A^{\nabla}=\frac{1}{det(A)}adj(A)=adj(A)$, which is also definite.\end{rem}

\begin{proof} The diagonal entries in~$adj(A)$ are  sums of cycle tracks of $A$, and thus the Id summand~$1_R$ dominates every diagonal entry. Also, by Theorem 2.17(ii), we have  $$1_R=det(A^{\nabla}),$$ as required for  definite matrices. 
\end{proof}

\begin{pro} $ adj(AB) \models_{gs} adj(B) adj(A)$.\end{pro}

\noindent \textbf{Proof.} ~\cite[Proposition 4.8]{STMA}.
\vskip 0.5 truecm

Recall the classical Bruhat (LDU) decomposition, whose tropical analog in ~\cite{LDM} is called the LDM decomposition.

\begin{lem}

$ $

\noindent a. If $A$ is a not strictly singular triangular matrix over $\mathbb{T}$ (respectively non-singular over~$R$), then $A$ is elementarily factorizable.

\noindent b. If $A$ is a not strictly singular matrix over $\mathbb{T}$ (respectively non-singular over~$R$), then~$A^{\nabla}$ (respectively  $A^{\nabla\nabla}$) is elementarily factorizable. 
\end{lem}

\begin{proof} See the $LDM$ decomposition in ~\cite{LDM}, or an alternative proof in ~\cite[Lemma 6.5 and Corollary 6.6]{FTM}.\end{proof}

One can find a cruder  factorization in ~\cite{TLI}, where sufficient conditions are established for~$trop(AB)=trop(A)trop(B)$, when looking at the tropical structure as the image of a valuation over the field of Puiseux series,  using the classical $LDU$ decomposition over this field. 
In section 3 we will show the connection of definite matrices and factorization to the well known tropical closure operation $*$, studied in ~\cite{KS} and ~\cite{visual}.

\vskip 0.5 truecm

\subsection{Polynomials}

$ $

As one can see in ~\cite{STA}, the polynomials over the tropical structure are rather straightforward to view geometrically.  We notice that the graph of a monomial $a_ix^i\in \mathbb{T}[x]$  is a line, where the power $i$  indicates the slope. Since $\mathbb{T}$ is ordered we may present its elements on an axis, directed rightward, where if $a<b$ then $a$ appears left to $b$ on the~$\mathbb{T}$-axis, for every pair of distinct elements $a,b\in T$. It is now easy to understand that a tropical polynomial $$\sum_{i=0}^n a_ix^i\in R[x]$$ takes the value of the dominant monomial among $a_ix^i$ along the $\mathbb{T}$-axis. That having been said, it is possible that some monomials in the polynomial would not dominate for any $x\in \mathbb{T}$.

\begin{df}
Let $$f(x)=\sum_{i=0}^n a_ix^i\in R[x]$$ be a supertropical polynomial. We call monomials in $f(x)$ that dominate for some~$x\in R$ \textbf{essential}, and monomials in $f(x)$ that do not dominate for any $x\in R$ \textbf{inessential}. We write $$f^{es}(x)=   \sum_{i\in I} a_ix^i\in R[x],$$ where $a_ix^i$ is an essential monomial $ \forall i\in I$, called \textbf{the essential polynomial} of $f$.
 \end{df}

\begin{df}
We say that $b$ is a $k-$\textbf{root} of $a$, for some $k\in \mathbb{N}$, denoted as~$b= \sqrt[k]{a}$, if~$b^k=a$.
\end{df}

\begin{rem}
If $a,b\in T$ and $a^k=b^k$ then $a^k+b^k=(a+b)^k\in G$, and therefore~$a+b\in G$ and $a=b$. That is, the $k$-root of a tangible element is unique.
 \end{rem}

\begin{df}
We call an element $r \in R$ a \textbf{root} of a polynomial $f(x)$ if $$f(r)\models_{gs}0_R.$$
\end{df}

We distinguish between two kinds of roots of supertropical polynomials.

\begin{df} We refer to roots of a polynomial being obtained as an intersection of two leading tangible monomials as \textbf{corner roots}, and to  roots that are being obtained from one leading  ghost monomial as \textbf{non-corner roots}.
\end{df}

\vskip 0.25 truecm

\begin{rem} Suppose $f(x)=\sum a_ix^i \in R[x]$. We specialize to elements $r\in R$, starting with $r$ small and then increasing.

\noindent 1. The constant term $a_0$ and the leading monomial $a_nx^n$  dominate first and last, respectively, due to their slopes. Furthermore, they are the only ones that are \textbf{necessarily} essential in every polynomial.

\noindent 2. The intersection of an essential monomial $a_ix^i$ and the next essential monomial~$a_jx^j$ where $j>i$, is the $i^{th}$ root of $f$ (counting multiplicities), denoted as $\alpha _i=\sqrt[k]{\frac{a_i}{a_j}}$, and is of multiplicity~$k=j-i$. 

\noindent Proof: The monomial $a_ix^i$ dominates all monomials between $a_ix^i$ and $a_jx^j$. Therefore, when $r\in (\alpha_{i-1},\alpha_{i}],$
$$f(r)=a_jr^j+a_ir^i=a_j\left(r^j+\frac{a_i}{a_j}r^i\right)\in G \Rightarrow \left(r^j+\frac{a_i}{a_j}r^i\right)=r^i\left(r^{j-i}+\frac{a_i}{a_j}\right)\in G$$ which means  $(r+\alpha_i)^k\in G$, and therefore $\alpha_i$ is a root of $f$ with multiplicity $k$.

\end{rem}

\vskip 0.25 truecm

\vskip 0.5 truecm

\subsection{Supertropical characteristic polynomials and eigenvalues.}

$ $

 We follow the description  in ~\cite[\S 5]{STMA2}.

\begin{df}
$\forall v \in T^n$ and $A \in M_n(R)$ such that $\exists \alpha \in T\bigcup\{0_R\}$ where~$ Av\models_{gs} \alpha v$, we say that $v$ is a \textbf{supertropical eigenvector}  of $A$ with a \textbf{supertropical eigenvalue}~$\alpha$.

\noindent The \textbf{characteristic polynomial} of $A$ (also called the maxpolynomial) is defined to be~$f_A(x)=det(xI+A)$. The tangible value of the roots of the characteristic polynomial~$f_A$ are the eigenvalues of $A$, as shown in ~\cite[Theorem 7.10]{STMA}. The coefficient of $x^k$ in this polynomial is a sum of the tropical determinants of all $n-k\times n-k$ minors, obtained by deleting $k$ chosen rows of the matrix, and their corresponding columns. These minors are defined as \textbf{principal sub-matrices}. 
\end{df}
The combinatorial motivation for the tropical characteristic polynomial is the Best Principal Submatrix problem, and has been studied by Butkovic in ~\cite{CCP} and ~\cite{ETCP}.

\begin{thm} (Supertropical Hamilton-Cayley) Any matrix $A$ satisfies its characteristic polynomial,  in the sense that $f_A(A)$ is ghost. \end{thm}

\begin{proof} See ~\cite{CH} or ~\cite[Theorem 5.2]{STMA} for a supertropical statement and proof. \end{proof}

\begin{pro}
If $ \alpha \in T\bigcup{0_R}$ is a supertropical eigenvalue of a matrix $A \in M_n(T)$, then $\alpha^i$ is a supertropical eigenvalue of  $A^i$.
\end{pro}

\begin{proof}

See ~\cite[Proposition 3.2]{CP}.
\end{proof}

\vskip 0.5 truecm

However, we notice that $\{\alpha^i:\alpha\quad \mbox{is an eigenvalue of a matrix A}\}$ need not be the only supertropical eigenvalues of the matrix $A^i$, as shown in the next example.

\begin{exa}

Consider the $2\times 2$ matrix

$$A=
\left(
\begin{array}{cc}
0 & 0\\
1 & 2
\end{array}
\right).$$

\vskip 0.25 truecm

\noindent{Then $f_A(x)=x^2+2x+2 \Rightarrow f_A(x) \in G\quad \mbox{when}\quad x=0,2$.}

\vskip 0.25 truecm

\noindent{However,}

\vskip 0.25 truecm

$$A^2=
\left(
\begin{array}{cc}
1 & 2\\
3 & 4
\end{array}
\right).$$

\vskip 0.25 truecm

\noindent Thus $f_{A^2}(x)=x^2+4x+5^{\nu} \Rightarrow f_{A^2}(x) \in G \quad \mbox{when}\quad x^{\nu}\leq 1^{\nu} \quad \mbox{or}\quad x=4$.
\end{exa}

\vskip 0.25 truecm

\begin{thm}
Let $A$ be in $M_n(R)$. If $$f_{A}(x)=\sum_{i=0}^n \alpha_ix^i\ \text{ and }\ f_{A^m}(x)=\sum_{i=0}^n \beta _ix^i$$ are the characteristic polynomials  of $A$ and its $m^{th}$ power, respectively, then~$$f_{A^m}(x^m) \models_{gs} f_A(x)^m.$$
\end{thm}

\begin{proof}

See ~\cite[Theorem 3.6]{CP}.
\end{proof}

\begin{cor}

$ $

\noindent a. If  $f_{A^m}\in T[x]$ then equality holds in Theorem 2.31.

\noindent b. Every corner root of $f_{A^m}$ is an $m^{th}$ power of a corner root of $f_A$.
\end{cor}

\begin{proof}

See ~\cite[Corollary 3.7 and Corollary 3.10]{CP}.\end{proof}

\vskip 0.5 truecm

\section{The $\nabla$ of a definite matrix}

As shown in ~\cite{FTM} and ~\cite{KS}, reductions to definite matrices can simplify verifications of some complicated properties of matrices. On top of that, matrices of this form   enjoy some  interesting properties of their own.

\begin{rem} Since ghost equivalent implies equality in the tropical setting, we will formulate our results in terms of ghost equivalence (i.e. over $R$), with the understanding that the corresponding tropical equalities (i.e. over $\mathbb{T}$) follow automatically.\end{rem}

\subsection{Stabilization under the $\nabla$ operation}
 
$ $

The goal in this section is to show that $\nabla$ is a closure operation.

\begin{lem}
 
$ $

\noindent (i) $P^\nabla =P^{-1}$ whenever~$P$ is an invertible matrix. 

\noindent (ii) det(PA)=det(P)det(A)=det(AP), for $P$ an invertible matrix.

\noindent (iii) $(PA)^\nabla =A^\nabla P^\nabla$ where~$det(A)$ is invertible and~$P$ is an invertible matrix.

\noindent (iv) Let~$\bar{A}$ be the left definite form of the matrix $A$, i.e. $A=P\bar{A}$ for some invertible matrix $P$. Then~$A^\nabla =\bar{A}^{\nabla}P^{-1}$.
\end{lem}

\begin{proof} For (i), (iii) and (iv) see ~\cite[Lemma 5.7]{FTM}.

\noindent(ii) We recall that for any permutation $\sigma$ we have a corresponding permutation $\rho$ such that $\sigma(i)=\rho(j),\ \sigma(j)=\rho(i)\ \text{and}\ \sigma(k)=\rho(k)\ \forall k\ne i,j.$

As a result, if $P=E_{i,j}$ is a transposition matrix and $A=(a_{i,j})$, then $$det(E_{i,j}A)=\sum_{\sigma} a_{i,\sigma(j)}a_{j,\sigma(i)}\prod_{k\ne j,i}a_{k,\sigma(k)}=\sum_{\rho} a_{i,\rho(i)}a_{j,\rho(j)}\prod_{k\ne j,i}a_{k,\rho(k)}=det(A)=det(E_{i,j})det(A).$$

\noindent  If $P=E_{\alpha\cdot j}$ is a diagonal multiplier and $A=(a_{i,j})$, then $$det(E_{\alpha\cdot j}A)=\sum_{\sigma} \alpha\cdot a_{j,\sigma(j)}\prod_{k\ne j}a_{k,\sigma(k)}=\alpha\sum_{\sigma}  a_{j,\sigma(j)}\prod_{k\ne j}a_{k,\sigma(k)}=\alpha det(A)=det(E_{\alpha\cdot j})det(A).$$ Inductively the claim holds for every invertible matrix $P$.

 \end{proof}

\begin{cla} If $A$ is a definite matrix, then $A^\nabla A\cong_{\nu} A^\nabla\cong_{\nu} AA^\nabla$.\end{cla}

\begin{proof} See  ~\cite{MCS} for general semirings, or ~\cite[Claim 6.1]{FTM} for the special case of the  supertropical semiring. \end{proof}

\begin{cor} Let $A$ be a matrix with  left definite form $\bar{A}$ and  right definite form~$\tilde{A}$, i.e. $A=P\bar{A}=\tilde{A}Q$ for  invertible matrices $P$ and $Q$. Then

\noindent a. $\bar{A}^{\nabla\nabla} \cong_{\nu} \bar{A}^\nabla\cong_{\nu}I'_A,\ \tilde{A}^{\nabla\nabla}\cong_{\nu}\tilde{A}^{\nabla}\cong_{\nu}I_A$.

\noindent b. $A^{\nabla\nabla} \cong_{\nu} PA^\nabla P$.

%\noindent c. $\bar{A}^{\nabla\nabla} \models_{gs} \bar{A},$ and therefore ${A}^{\nabla\nabla} \models_{gs} {A}.$

\end{cor}

\begin{proof}

$ $

\noindent a. According to ~\cite[Corollary 4.4]{STMA2} we know that $A^\nabla \cong_{\nu}A^\nabla A^{\nabla\nabla} A^\nabla$. By applying Claim 3.3 to $\bar{A}^{\nabla}$ we can conclude
$$\bar{A}^\nabla \cong_{\nu}\bar{A}^\nabla (\bar{A}^{\nabla\nabla} \bar{A}^\nabla) \cong_{\nu} \bar{A}^\nabla \bar{A}^{\nabla\nabla} \cong_{\nu} \bar{A}^{\nabla\nabla}.$$ According to Claim 3.3 $$\bar{A}^\nabla\cong_{\nu}\bar{A}^\nabla\bar{A}=\bar{A}^\nabla P^{-1}P\bar{A}=A^\nabla A=I_A',\ \text{and}\ \tilde{A}^\nabla\cong_{\nu}\tilde{A}\tilde{A}^\nabla= \tilde{A}QQ^{-1}\tilde{A}^\nabla= AA^\nabla=I_A$$

\noindent b. By Lemma 3.2 we have $$A^{\nabla\nabla}=(P\bar{A})^{\nabla\nabla}=P\bar{A}^{\nabla\nabla}\cong_{\nu} P\bar{A}^{\nabla}=PA^\nabla P.$$

%\noindent c. Using Claim 3.3.~ for $\bar{A}^{\nabla\nabla}$ and $\bar{A}^\nabla$ we have that

%$$\bar{A}^{\nabla\nabla}\cong_{\nu}\bar{A}^{\nabla\nabla}\bar{A}^{\nabla}\cong_{\nu}\bar{A}^{\nabla\nabla}\bar{A}^{\nabla}\bar{A}\cong_{\nu} I_{\bar{A}^{\nabla}}\bar{A}\models_{gs}\bar{A}.$$ Thus $${A}^{\nabla\nabla}=P\bar{A}^{\nabla\nabla}\models_{gs}P\bar{A}=A.$$

\end{proof}

\begin{thm} 
We denote  the application  $k$ times of $\nabla$ to $A$ as $A^{\nabla^{(k)}}$. Then  $$A^{\nabla^{(k)}}\cong_{\nu}A^{\nabla^{(k+2)}}\ \forall  k\geq 1.$$  That is,  applying $\nabla$  to $A^\nabla$ is an operator of order 2, which means that the $\nabla$ operation on the set $\{A^\nabla:A\in M_n(R)\}$ acts like the inverse operation.
\end{thm}

\begin{proof}
In the same way as in Corollary 3.4 (b) we see that $$(A^{\nabla})^{\nabla\nabla}=(\bar{A}^{\nabla}P^{-1})^{\nabla\nabla}=(P\bar{A}^{\nabla\nabla})^\nabla\cong_{\nu}(P\bar{A}^{\nabla})^{\nabla}=\bar{A}^{\nabla\nabla}P^{-1}\cong_{\nu}\bar{A}^{\nabla}P^{-1}=A^{\nabla}.$$ The general case then follows inductively.
\end{proof}

In summary, $\nabla\nabla$ is a sort of closure operation, which yields elementarily factorizable matrices.

\vskip 0.25 truecm

\begin{exa}

$ $

\noindent (i) For every non-singular $2\times 2$ matrix 
$$A=
\left(
\begin{array}{cc}
a & b\\
c & d
\end{array}
\right)$$ we have that 
$$A^\nabla=det(A)^{-1}
\left(
\begin{array}{cc}
d & b\\
c & a
\end{array}
\right)$$ and
$$A^{\nabla\nabla}=
\left(
\begin{array}{cc}
a & b\\
c & d
\end{array}
\right)=A,$$ imply that the formula in Theorem 3.5~holds in the $2\times 2$ case for all $k\geq 0$. We conclude that~$A$ is elementarily factorizable.
\vskip 0.25 truecm

\noindent (ii) Let 
$$A=
\left(
\begin{array}{ccc}
0 & 0& -\\
- & 0 & 0\\
1 & - & 0
\end{array}
\right),$$ (where $-$ denotes $-\infty$).

Thus 
 $$A^\nabla=
\left(
\begin{array}{ccc}
-1 & -1 & -1\\
0 & -1 & -1 \\
0 & 0 & -1
\end{array}
\right),$$$$
A^{\nabla\nabla}=
\left(
\begin{array}{ccc}
0 & 0 & -1^\nu\\
0^\nu & 0 & 0 \\
1 & 0^\nu & 0
\end{array}
\right),\text{ which  is different than } A^\nabla, $$

$$A^{\nabla^{(3)}}=
\left(
\begin{array}{ccc}
-1 & -1 & -1\\
0 & -1 & -1 \\
0 & 0 & -1
\end{array}
\right)=A^\nabla\text{ and  therefore equal }A^{\nabla^{(2k-1)}}\ \forall k\in \mathbb{N},$$ and
$$A^{\nabla^{(4)}}=
\left(
\begin{array}{ccc}
0 & 0 & -1^\nu\\
0^\nu & 0 & 0 \\
1 & 0^\nu & 0
\end{array}
\right)=A^{\nabla\nabla}\text{ and  therefore equal }A^{\nabla^{(2k)}}\ \forall k\in \mathbb{N}.$$

\end{exa}
\vskip 0.5 truecm

\subsection{The  closure operation $*$ and   power stabilization}

$ $

Noticing that $A^{\nabla}$ arises from supertropical algebraic considerations (as its product with $A$ gives a pseudo-identity), we would like to make a connection to the familiar tropical concept of the Kleene star, denoted as $A^*$, which has been widely studied since the~60's. In the next theorem and proposition we give  results relevant to ~\cite{YL} and ~\cite{KS}, obtained over the tropical setting, with the understanding that tropical equality implies supertropical ghost-equivalent. Since the definition for $A^*$  requires that~$det(A)$ is bounded by $1_R$,  we  consider here the special case of definite matrices. According to Lemma~2.20(a), we can conclude from the LDM factorization of $A^*$ in ~\cite{T&I} that $A^*$ is elementarily factorizable. 

% and almost simultaneously 

\begin{thm} (See ~\cite[Theorem 2]{YL}) If~$A$ is an $n\times n$ definite matrix  and~$k$ is  a natural number,  then~$A^k\cong_{\nu}A^{k+1},\ \forall k\geq n-1$.

\end{thm}

\begin{pro}
$A^{\nabla}\cong_{\nu}A^*\cong_{\nu}A^{n-1}$ when $A$ is definite.  

\end{pro}

\begin{proof}
For $A^{\nabla}\cong_{\nu}A^*$ see ~\cite[Theorem 4]{YL}. As described in ~\cite{KS}, the equivalence to $A^{n-1}$ is immediate from the definition of $A^*$ for a definite matrix $A$: \begin{equation}A^*=\sum_{i\in \mathbb{N}\bigcup{0}} A^i.\end{equation} Indeed, according to ~\cite[Theorem 2]{YL} we have that
$$\widehat{A}\leq \widehat{A}^2\leq...\leq\widehat{A}^{n-1}\cong_{\nu}A^n\cong_{\nu}...$$ (due to the  diagonal, each position can only increase comparing to the corresponding position in the previous power and the value of each entry will stabilize at  power $n-1$ at most).\end{proof} 

Combining Corollary 3.4, Theorem 3.6 and Proposition 3.7  we get $$A^*\cong_{\nu}A^{n-1}\cong_{\nu}A^{\nabla}\cong_{\nu}A^{\nabla\nabla}\cong_{\nu}I_A,$$ for a definite matrix $A$ of order $n$.

\vskip 0.5 truecm

\section{An alternative proof of Theorem 2.10: $det(AB)\models_{gs} det(A)det(B)$}

The proof of Theorem 2.10  is given by means of a multilinear function. In the classical case there exists an alternative, somewhat  easier proof, using the factorization of non-singular matrices into the product of elementary matrices. Due to Lemma 2.20 we can now give a tropical  analog of this proof.

If~$A$ or~$B$ are strictly singular (without lost of generality~$B$), then according to ~\cite[Theorem 6.5]{STMA} the columns of~$B$, denoted as~$C_1,...,C_n$, are tropically dependent:~$$\sum \alpha_i C_i\in G\bigcup{0_R},\ \alpha_i\in T\forall i.$$ Denoting the rows of $A$ as $R_1,...,R_n$ and the columns of $AB$ as $c_1,...,c_n$ we get $$\sum \alpha_i c_i=
\left(
\begin{array}{ccc}
\sum \alpha_i R_1C_i\\ 
\vdots\\
 \sum \alpha_i R_nC_i
\end{array}\right)
\in G_{0_R}.$$ Therefore $det(AB)\in G\bigcup{0_R}$ as required.

In order to obtain the $\nabla$ operation for (not-strictly) singular matrices we work over the tropical semifield and will prove that $det(AB)$ equals $det(A)det(B)$ plus a term that is  attained twice. 

Using Theorem 2.17 we have that $det(A^\nabla)=det(A)^{-1}$, and it  suffices to prove the theorem for $A^\nabla ,B^\nabla$. By Lemma 3.2(ii), if $A$ or $B$ are  transposition  matrices  and diagonal multipliers then the theorem holds with equality. Due to Lemma 2.20(b) we only have to show the impact of Gaussian matrices. 

Let $E=E_{row\ i+\alpha\cdot row\ j}$ be a Gaussian matrix adding row $j$, multiplied by $\alpha$, to row~$i$. We denote the rows of $A$ as $R_1,...,R_n$. Then the rows of $EA$ are the same as $A$'s, except for the $i^{th}$ row which is $R_i+\alpha R_j$, therefore \begin{equation}det(EA)=\sum_{\sigma\in S_n}a_{1,\sigma(1)}\cdots a_{i,\sigma(i)}\cdots a_{n,\sigma(n)}+\sum_{\sigma\in S_n}a_{1,\sigma(1)}\cdots \alpha a_{j,\sigma(i)}\cdots a_{n,\sigma(n)},\end{equation} where the sum on the left is $det(A)$. In order to analyze  the sum on the right, we note that \begin{equation}\forall i,j\ \exists ! \sigma,\rho\in S_n:\ \sigma(i)=\rho(j) , \sigma(j)=\rho(i),\ \text{and}\ \sigma(k)=\rho(k)\ \forall k\ne i,j,\end{equation} and let $E$ fix $i$ and $j$. Thus 

\begin{equation}det(EA)=det(A)+\end{equation}$$\sum_{\sigma,\rho\in S_n}(a_{1,\sigma(1)}\cdots \alpha a_{j,\sigma(i)}\cdots a_{j,\sigma(j)}\cdots a_{n,\sigma(n)}+a_{1,\sigma(1)}\cdots  a_{j,\rho(j)}\cdots \alpha a_{j,\rho(i)}\cdots a_{n,\sigma(n)}),$$ for $\sigma, \rho$ as in (4.2). Therefore the sum is  attained twice, as required. If $E$ is multiplied on the right, then the proof is  analogous using column operation instead of row operations.

As a result we have 
$$|AB|=|(AB)^\nabla|^{-1}\models_{gs} |B^\nabla A^\nabla|^{-1}=|P_1^{-1}\bar{B}^\nabla \bar{A}^\nabla P_2^{-1}|^{-1}=|P_1||\bar{B}^\nabla \bar{A}^\nabla|^{-1}|P_2|\models_{gs}$$$$ |P_1||\bar{B}^\nabla|^{-1}| \bar{A}^\nabla|^{-1}|P_2|=|P_1^{-1}\bar{B}^\nabla|^{-1}| \bar{A}^\nabla P_2^{-1}|^{-1}=|B^\nabla|^{-1}|A^\nabla|^{-1}=|A||B|,$$ where $P_1,P_2$ are the right conductor of $B$ and left conductor of $A$ respectively, and as shown in Lemma 2.20, $\bar{A}$ and $\bar{B}$ are products of Gaussian matrices, for which the theorem holds according to (4.3).

\vskip 0.25 truecm

\section{Tropically conjugate and similar matrices}

Although our factorization results for $A^\nabla$ follow from those for $A^*$, the $\nabla$ operation has additional applications in representation theory.

\begin{df} A  matrix $B$ is \textbf{tropically conjugate} to $B'$ if there exists a non-singular matrix $A$ such that \begin{equation}A^{\nabla}BA\models_{gs} B'.\end{equation} If equality holds in (5.1) then we say that $B$ is \textbf{tropically similar} to $B'$.\end{df}

\begin{rem} If $B$ is tropically similar to $B'$, then

\noindent a. $det(B')\models_{gs} det(B)$, with equality when $B'$ is non-singular.

\noindent b. $tr(B')\models_{gs} tr(B)$.

\end{rem}
\vskip 0.25 truecm

These relations can be seen by direct computation, but are also a consequence of Theorem 5.4 below.
\vskip 0.15 truecm

Following the next motivating example, we have an analog to the  theorem on similar matrices in classical algebra.
\vskip 0.25 truecm

\begin{exa}

$ $
\vskip 0.25 truecm

\noindent 1. Let $$A=\left(
\begin{array}{cc}
2 & 0\\
1 & 0
\end{array}
\right),\ B=\left(
\begin{array}{cc}
1 & 2\\
3 & 1
\end{array}
\right),\ B'=\left(
\begin{array}{cc}
3 & 1^\nu\\
5 & 3
\end{array}
\right).$$ We have $A^\nabla BA=B'$,  and therefore $f_{B'}(x)=x^2+3^{\nu}x+6^{\nu}\models_{gs} f_{B}(x)=x^2+1^{\nu}x+5$. As a result all supertropical eigenvalues of $B$ are supertropical eigenvalues of $ A^\nabla BA$.
\vskip 0.5 truecm

\noindent 2. By taking $B=\left(\begin{array}{cc}0 & 0 \\ 1 & 2 \end{array}\right),\ B'=\left(\begin{array}{cc}1 & 3 \\ 1 & 2 \end{array}\right), A=\left(\begin{array}{cc}0 & 1^{\nu} \\ -\infty & 0 \end{array}\right)$ we get $$A^{\nabla}BA=\left(
\begin{array}{cc}
2^\nu & 3^\nu\\
1 & 2^\nu
\end{array}
\right)\models_{gs} B'.$$ However, $|B|=2\ \text{and}\ |B'|=4$ do not ghost-surpass each other, which shows that there need not be a ghost surpassing relation between the characteristic polynomials of tropically  conjugate matrices.
\end{exa}

\vskip 0.25 truecm

\begin{thm}
If $B$ is tropically similar to $B'$, then $f_{B'}\models_{gs} f_B$. 

\end{thm}

\begin{proof}

 Let $B'=A^{\nabla}BA$, where  $A$ is a non-singular matrix.

 \vskip 0.15 truecm

Let~$\bar{A}$ be the right definite form of~$A$ and~$E$ be its right conductor  which means~$$A=\bar{A}E\ \text{and}\ A^{\nabla}=E^{-1}\bar{A}^{\nabla}.$$ 
\vskip 0.15 truecm

Using Lemma 3.2, we notice that $A^{\nabla}BA$ and $\bar{A}^{\nabla}B\bar{A}$ share the same characteristic polynomial, since $$det(xI+A^{\nabla}BA)=det(xI+E^{-1}\bar{A}^{\nabla}B\bar{A}E)=$$$$det(E)^{-1}det(xI+\bar{A}^{\nabla}B\bar{A})det(E)=det(xI+\bar{A}^{\nabla}B\bar{A}).$$ Therefore it is sufficient to prove this theorem for the definite form of $A$, which according to Proposition~3.8 coincides with its Kleene star and ${n-1}$ power.

We consider $A$ to be a definite matrix. We denote the characteristic polynomials of~$B'$ and $B$ as $$f_{B'}(x)=x^n+\sum_{k=0}^{n-1} \beta'_k x^{n-k},\ \ \ f_{B}(x)=x^n+\sum_{k=0}^{n-1} \beta_k x^{n-k}$$respectively, and the entries of the matrices as $$A=(a_{i,j}),\ \ A^\nabla=A^*=A^{n-1}=A^{n}=(a'_{i,j})\ \text{ and }\ B=(b_{i,j}).$$

We  show that $\beta'_k\models_{gs} \beta_k\ \forall k$. That is, every summand in $\beta'_k$   either has an identical summand  in $\beta_k$ or appear twice, creating a ghost summand.
\vskip 0.2 truecm

\noindent The entry of $A^\nabla BA$ appears in the  $(i,\sigma(i))$ position is $\sum_{t,s=1}^n a'_{i,t}b_{t,s}a_{s,\sigma(i)}.$
\vskip 0.2 truecm

\noindent The summands in $\beta'_k$ (the sum of determinants of $k\times k$ principal sub-matrices)  are  \begin{equation}\prod_{j=1}^k a'_{i_j,t_j}b_{t_j,s_j}a_{s_j,\sigma(i_j)},\end{equation} where $\sigma\in S_k$, and $t_j,s_j\in\{1,...,n\}$.
\vskip 0.2 truecm

We can factor each permutation  into disjoint cycles $(i_j\ \sigma(i_j)\ \sigma^2(i_j)\ \cdots\ \sigma^{-1}(i_j))$, and we denote $i_j$ as~$j$.  We abbreviate $\ \sigma(j), \sigma^2(j),...,\sigma^{-1}(j)\ $ to $\ \sigma, \sigma^2, ... , \sigma^{-1}$.  

\noindent Therefore, each permutation track can be factored into cycle tracks:
\begin{equation}\prod_{\tiny{\begin{array}{c} \text{cycles}\\ \text{of}\ \sigma\end{array}}} (a'_{j,t_j}b_{t_j,s_j}a_{s_j,\sigma})(a'_{\sigma,t_\sigma}b_{t_\sigma,s_\sigma}a_{s_\sigma,\sigma^2}) \cdots(a'_{\sigma^{-1},t_{\sigma^{-1}}}b_{t_{\sigma^{-1}},s_{\sigma^{-1}}}a_{s_{\sigma^{-1}},j}),\end{equation} obtaining a product of brackets, where the left indices are distinct, describing a permutation. Changing the location of the brackets does not change the value of this term. We   use this relocation  of brackets to show that identical terms are being obtained on different permutation tracks.
\vskip 0.2 truecm

 The term $a'_{j,t_j}$ is a sum arising from the Kleene star: $$a'_{j,t_j}=\sum_{_{r_{_1},...,r_{_{n-2}}=1}}^n(a_{j,r_1}a_{r_1,r_2}\cdots a_{r_{n-2},t_j}).$$
 
Thus, we write  (5.3) as 

$$\prod_{\tiny{\begin{array}{c} \text{cycles}\\ \text{of}\ \sigma\end{array}}} 
(a_{j,r_{j,1}}a_{r_{j,1},r_{j,2}}\cdots a_{r_{j,n-2},t_j}b_{t_j,s_j}a_{s_j,\sigma})
(a_{\sigma,r_{\sigma,1}}a_{r_{\sigma,1},r_{\sigma,2}}\cdots a_{r_{\sigma,n-2},t_\sigma}b_{t_\sigma,s_\sigma}a_{s_\sigma,\sigma^2}) \cdots$$$$(a_{\sigma^{-1},r_{\sigma^{-1},1}}a_{r_{\sigma^{-1},1},r_{\sigma^{-1},2}}\cdots a_{r_{\sigma^{-1},n-2},t_{\sigma^{-1}}}b_{t_{\sigma^{-1}},s_{\sigma^{-1}}a_{s_{\sigma^{-1}}},j}).$$

\vskip 0.2 truecm

Next, we divide our proof into disjoint cases as {follows}:
$$\begin{array}{c}
                                                                                                                                \large\text{\framebox{$a'_{j,t_j}b_{t_j,s_j}a_{s_j,\sigma(j)}$}}\\
                                                                                                                             \normalsize \swarrow\ \ \ \ \ \ \ \ \ \ \ \ \ \ \ \ \ \ \ \ \ \ \ \ \ \searrow\\
\ \ \text{\framebox{$\begin{array}{c}\text{Case 1}\\ s_j=\sigma(j)\ \forall j\end{array}$}}\ \ \ \ \ \ \ \ \ \ \ \ \ \ \ \ \ \ \ \ \ \ \ \ \ \ \ \ \ \ \ \ \ \text{\framebox{$\begin{array}{c}\text{Case 2}\\ \exists j:\ s_j\ne \sigma(j)\end{array}$}} \\
                                                                     \ \      \swarrow\ \ \ \ \ \ \ \ \ \ \ \ \ \ \ \searrow\ \ \ \ \ \ \ \ \ \ \ \ \ \ \ \ \ \ \ \ \ \ \ \ \ \ \ \ \ \ \ \ \ \ \ \ \ \ \ \ \ \ \ \ \ \ \ \ \ \ \ \   \ \ \  \\
\text{\framebox{$\begin{array}{c}\text{Case 1.1}\\  j=t_j\ \forall j\\ \Rightarrow (5.2)\text{ is }\\ \prod b_{j,\sigma(j)}\end{array}$}}\ \ \ \ \text{\framebox{$\begin{array}{c}\text{Case 1.2}\\  \exists j:\ i_j\ne t_j\\ \Rightarrow (5.2)\text{ is }\\ \prod a'_{j,t_j}b_{t_j,\sigma(j)}\end{array}$}}      \ \ \ \ \ \ \ \ \ \ \ \ \ \ \ \ \ \ \ \ \ \ \ \ \ \ \ \ \ \ \ \ \ \   \ \ \ \ \ \ \ \ \ \ \ \ 
\end{array}$$
\vskip 0.2 truecm

\noindent In \textbf{case 1}, $s_j=\sigma(j)\ \forall j$. Thus  $a_{s_j,\sigma(j)}$ is diagonal and therefore equals $0$ for every $j$.

\noindent Next, if we have the  special \textbf{sub-case 1.1}, $ j=t_j\ \forall j$, then  $a'_{j,t_j}$  are also diagonal entries, equal to $0$. Therefore (5.2) is $\prod b_{j,\sigma(j)}$, which is a summand in~$f_B$.

At this point we have established $\beta'_k= \beta_k+g_k,$ where $g_k\in R\ \ \forall k.$ It remains to prove that~$g_k\in G\ \ \forall k$, that is every other summand is attained at least twice, creating a ghost summand.

\noindent In \textbf{sub-case 1.2}, stating that $\exists j:\ i_j\ne t_j$, some $a'_{j,t_j}$ is not diagonal. Thus (5.3) is $$\prod_{\tiny{\begin{array}{c} \text{cycles}\\ \text{of}\ \sigma\end{array}}} 
(a_{j,r_{j,1}}a_{\underline{r_{j,1}},r_{j,2}}\cdots a_{r_{j,n-2},t_j}b_{t_j,\sigma})
(a_{\sigma,r_{\sigma,1}}a_{\underline{r_{\sigma,1}},r_{\sigma,2}}\cdots a_{r_{\sigma,n-2},t_\sigma}b_{t_\sigma,\sigma^2}) \cdots$$$$(a_{\sigma^{-1},r_{\sigma^{-1},1}}a_{\underline{r_{\sigma^{-1},1}},r_{\sigma^{-1},2}}\cdots a_{r_{\sigma^{-1},n-2},t_{\sigma^{-1}}}b_{t_{\sigma^{-1}},j}).$$

\noindent We analyze the second index of each element: $r_{j,1},\ j=1,...,k$.

\vskip 0.2 truecm

 If all these  indices are distinct, then $r_{\sigma^l,1}\mapsto r_{\sigma^{l+1},1}$ describes a permutation. Its track is obtained by moving the  brackets  one position \textbf{forward}:
 $$\prod_{\tiny{\begin{array}{c} \text{cycles}\\ \text{of}\ \sigma\end{array}}} 
a_{j,r_{j,1}})(a_{r_{j,1},r_{j,2}}\cdots a_{r_{j,n-2},t_j}b_{t_j,\sigma}
a_{\sigma,r_{\sigma,1}})(a_{r_{\sigma,1},r_{\sigma,2}}\cdots a_{r_{\sigma,n-2},t_\sigma}b_{t_\sigma,\sigma^2} \cdots$$$$a_{\sigma^{-1},r_{\sigma^{-1},1}})(a_{r_{\sigma^{-1},1},r_{\sigma^{-1},2}}\cdots a_{r_{\sigma^{-1},n-2},t_{\sigma^{-1}}}b_{t_{\sigma^{-1}},j}=$$

 $$\prod_{\tiny{\begin{array}{c} \text{cycles}\\ \text{of}\ \sigma\end{array}}} 
(a_{r_{j,1},r_{j,1}}a_{r_{j,1},r_{j,2}}\cdots a_{r_{j,n-2},t_j}b_{t_j,\sigma}
a_{\sigma,r_{\sigma,1}})(a_{r_{\sigma,1},r_{\sigma,1}}a_{r_{\sigma,1},r_{\sigma,2}}\cdots a_{r_{\sigma,n-2},t_\sigma}b_{t_\sigma,\sigma^2}a_{\sigma^2,r_{\sigma^{2},1}}) $$$$\cdots(a_{r_{\sigma^{-1},1},r_{\sigma^{-1},1}}a_{r_{\sigma^{-1},1},r_{\sigma^{-1},2}}\cdots a_{r_{\sigma^{-1},n-2},t_{\sigma^{-1}}}b_{t_{\sigma^{-1}},j}a_{j,r_{j,1}}).$$

\vskip 0.25 truecm

If $r_{j,1}=r_{i,1}$ for some $j\ne i$, then:
\vskip 0.25 truecm

\noindent If it occurs \textit{in the same cycle}, we obtain a new permutation by factoring this cycle into two sub-cycles, using the points of repetition (we do not change any of the other cycles):

$$[(a_{j,\underline{r_{j,1}}}a_{\underline{r_{j,1}},r_{j,2}}\cdots a_{r_{j,n-2},t_j}b_{t_j,\sigma}) \cdots
(a_{\sigma^l,\underline{\underline{r_{j,1}}}}a_{\underline{\underline{r_{j,1}}},r_{\sigma^l,2}}\cdots a_{r_{\sigma^l,n-2},t_{\sigma^l}}b_{t_{\sigma^l},\sigma^{l+1}}) \cdots$$$$
(a_{\sigma^{-1},r_{\sigma^{-1},1}}a_{r_{\sigma^{-1},1},r_{\sigma^{-1},2}}\cdots a_{r_{\sigma^{-1},n-2},t_{\sigma^{-1}}}b_{t_{\sigma^{-1}},j})]=$$
$$[(a_{j,\underline{r_{j,1}}}a_{\underline{\underline{r_{j,1}}},r_{\sigma^l,2}}\cdots a_{r_{\sigma^l,n-2},t_{\sigma^l}}b_{t_{\sigma^l},\sigma^{l+1}}) \cdots(a_{\sigma^{-1},r_{\sigma^{-1},1}}a_{r_{\sigma^{-1},1},r_{\sigma^{-1},2}}\cdots a_{r_{\sigma^{-1},n-2},t_{\sigma^{-1}}}b_{t_{\sigma^{-1}},j})]$$$$\small
[(a_{\sigma^l,\underline{\underline{r_{j,1}}}}a_{\underline{r_{j,1}},r_{j,2}}\cdots a_{r_{j,n-2},t_j}b_{t_j,\sigma})\cdots (a_{\sigma^{l-1},{r_{\sigma^{l-1},1}}}a_{r_{\sigma^{l-1},1},r_{\sigma^{l-1},2}}\cdots a_{r_{\sigma^{l-1},n-2},t_\sigma^{l-1}}b_{t_\sigma^{l-1},\sigma^{l}})]$$
\normalsize  (the squared brackets indicate cycles).
\vskip 0.25 truecm

\noindent If an index $r_{j,1}$ repeats \textit{in two different cycles}, we obtain a new permutation by joining these cycles into a single cycle, using the points of repetition (we do not change any of the other cycles):
\small

$$[(a_{j,\underline{r_{j,1}}}a_{\underline{r_{j,1}},r_{j,2}}\cdots a_{r_{j,n-2},t_j}b_{t_j,\sigma})
\cdots(a'_{\sigma^{-1},t_{\sigma^{-1}}}b_{t_{\sigma^{-1}},j})]\cdot$$
$$[(a_{i,\underline{\underline{r_{j,1}}}}a_{\underline{\underline{r_{j,1}}},r_{i,2}}\cdots a_{r_{i,n-2},t_i}b_{t_i,\sigma(i)})
\cdots(a'_{\sigma^{-1}(i),t_{\sigma^{-1}(i)}}b_{t_{\sigma^{-1}(i)},i})]=$$

$$[(a_{j,\underline{r_{j,1}}}a_{\underline{\underline{r_{j,1}}},r_{i,2}}\cdots a_{r_{i,n-2},t_i}b_{t_i,\sigma(i)})
\cdots(a'_{\sigma^{-1}(i),t_{\sigma^{-1}(i)}}b_{t_{\sigma^{-1}(i)},i})$$$$(a_{i,\underline{\underline{r_{j,1}}}}a_{\underline{r_{j,1}},r_{j,2}}\cdots a_{r_{j,n-2},t_j}b_{t_j,\sigma})
\cdots(a'_{\sigma^{-1},t_{\sigma^{-1}}}b_{t_{\sigma^{-1}},j})]$$
\normalsize (the squared brackets indicate cycles).
\vskip 0.25 truecm

\noindent In \textbf{case 2}, $\exists j:\ s_j\ne \sigma(j)$, and thus some $a_{s_j,\sigma(j)}$ is not diagonal. 
\vskip 0.15 truecm

We obtain the second permutation track using the algorithm of case 1.2,  analyzing the  indices $s_j,\ j=1,...,k$.

\noindent  If all these  indices are distinct, then $s_{\sigma^l}\mapsto s_{\sigma^{l+1}}$ describes a permutation. Even if this permutation coincides with $\sigma$, you still obtain a different summand of (5.2). Its track is obtained by moving the  brackets  one position \textbf{backwards}:

$$\prod_{\tiny{\begin{array}{c} \text{cycles}\\ \text{of}\ \sigma\end{array}}} a'_{j,t_j}b_{t_j,s_j})(a_{s_j,\sigma}a'_{\sigma,t_\sigma}b_{t_\sigma,s_\sigma}) \cdots(a_{s_\sigma^{-2},\sigma^{-1}} a'_{\sigma^{-1},t_{\sigma^{-1}}}b_{t_{\sigma^{-1}},s_{\sigma^{-1}}})(a_{s_{\sigma^{-1}},j}=$$
$$\prod_{\tiny{\begin{array}{c} \text{cycles}\\ \text{of}\ \sigma\end{array}}} (a_{s_{\sigma^{-1}},j}a'_{j,t_j}b_{t_j,s_j})(a_{s_j,\sigma}a'_{\sigma,t_\sigma}b_{t_\sigma,s_\sigma}) \cdots(a_{s_\sigma^{-2},\sigma^{-1}}a'_{\sigma^{-1},t_{\sigma^{-1}}}b_{t_{\sigma^{-1}},s_{\sigma^{-1}}}).$$
\vskip 0.15 truecm

\noindent(\small{This case is dual to the previous argument in the sense that 1.2 \textit{starts} where the entries of~$A$ are diagonal, and is \textit{being repeated} by a permutation track with no such restriction.})\normalsize
\vskip 0.2 truecm

\normalsize
\noindent In all  remaining summands  $s_j=s_i$, for some $j\ne i$. Then like in case 1.2, we factor and disjoint cycles, using the points of repetition, without changing the non-relevant cycles:
$$[ (a'_{j,t_j}b_{t_j,\underline{s_j}}a_{\underline{s_j},\sigma})\cdots(a'_{\sigma^l,t_{\sigma^l}}b_{t_{\sigma^l},\underline{\underline{s_j}}}a_{\underline{\underline{s_j}},\sigma^{l+1}}) \cdots(a'_{\sigma^{-1},t_{\sigma^{-1}}}b_{t_{\sigma^{-1}},s_{\sigma^{-1}}}a_{s_{\sigma^{-1}},j})]=$$
$$[ (a'_{j,t_j}b_{t_j,\underline{s_j}}a_{\underline{\underline{s_j}},\sigma^{l+1}}) \cdots(a'_{\sigma^{-1},t_{\sigma^{-1}}}b_{t_{\sigma^{-1}},s_{\sigma^{-1}}}a_{s_{\sigma^{-1}},j})]\cdot$$
$$[(a'_{\sigma^l,t_{\sigma^l}}b_{t_{\sigma^l},\underline{\underline{s_j}}}a_{\underline{s_j},\sigma})\cdots (a'_{\sigma^{l-1},t_{\sigma^{l-1}}}b_{t_{\sigma^{l-1}},s_{\sigma^{l-1}}}a_{s_{\sigma^{l-1}},\sigma^l})],$$

\noindent and

$$[ (a'_{j,t_j}b_{t_j,\underline{s_j}}a_{\underline{s_j},\sigma}) \cdots(a'_{\sigma^{-1},t_{\sigma^{-1}}}b_{t_{\sigma^{-1}},s_{\sigma^{-1}}}a_{s_{\sigma^{-1}},j})]\cdot$$
$$ [(a'_{i,t_i}b_{t_i,\underline{\underline{s_j}}}a_{\underline{\underline{s_j}},\sigma(i)})\cdots(a'_{\sigma^{-1}(i),t_{\sigma^{-1}(i)}}b_{t_{\sigma^{-1}(i)},s_{\sigma^{-1}(i)}}a_{s_{\sigma^{-1}(i)},i})]=$$

$$[ (a'_{j,t_j}b_{t_j,\underline{s_j}}a_{\underline{\underline{s_j}},\sigma(i)})\cdots(a'_{\sigma^{-1}(i),t_{\sigma^{-1}(i)}}b_{t_{\sigma^{-1}(i)},s_{\sigma^{-1}(i)}}a_{s_{\sigma^{-1}(i)},i})\cdot
$$$$ (a'_{i,t_i}b_{t_i,\underline{\underline{s_j}}}a_{\underline{s_j},\sigma}) \cdots(a'_{\sigma^{-1},t_{\sigma^{-1}}}b_{t_{\sigma^{-1}},s_{\sigma^{-1}}}a_{s_{\sigma^{-1}},j})].$$

\vskip 0.2 truecm

\noindent The assertion then follows.

\end{proof}

\begin{cor} If $B$ is tropically similar to $B'$, then

\noindent 1. Every supertropical eigenvalue of $B$ is a supertropical eigenvalue of $B'$.

\noindent 2. If $f_{B'}\in T[x]$ then $f_{B'}=f_B$.

\noindent 3.   $B$ satisfies $f_{B'}$ in the sense that $f_{B'}(B)$ is ghost..

\end{cor}

\begin{proof}

$ $

\noindent 1. $B$ is tropically similar to $B'$. If $\lambda$ is an eigenvalue of $B$, then $f_{B'}(\lambda)\models_{gs} f_{B}(\lambda)\in G$.

\noindent 2. Follows immediately from the theorem.

\noindent 3. By  Theorem 2.28  we get $f_{B'}(B)\models_{gs} f_{B}(B)\in G.$

\end{proof}

\section{The characteristic polynomial of the pseudo-inverse.}

Current research regarding the eigenspaces of a supertropical matrix shows that these spaces may  behave in an undesirable fashion, as shown in ~\cite[Example 5.7]{STMA2}. The following chapter provides properties regarding the  characteristic polynomial, eigenvalues, determinant and trace of $A^\nabla$, which may provide techniques for solving the dependency problem  of supertropical eigenspaces.
\vskip 0.25 truecm

We observe the following motivating example.
\vskip 0.25 truecm

\begin{exa}

Let $A=\left(
\begin{array}{ccc}
1 & 0 & - \\
3 & 4 & - \\
- & - & 1
\end{array}
\right)$ (where $-$ denotes $-\infty$). Then~$$f_A(x)=x^3+4x^2+5^\nu x+6$$ and $$A^\nabla=6^{-1}\left(
\begin{array}{ccc}
5 & 1 & - \\
4 & 2 & - \\
- & - & 5
\end{array}
\right)=\left(
\begin{array}{ccc}
-1 & -5 & - \\
-2 & -4 & - \\
- & - & -1
\end{array}
\right).$$ As a result we get $f_{A^\nabla}(x)=x^3+(-1)^\nu x^2+(-2) x-6$. Multiplying by $det(A)$ we get~$$6x^3+5^\nu x^2+4 x+0,$$ which has the same coefficients  of $f_A$ but in opposite order. Consequently, the inverse of every supertropical eigenvalue of $A$ is a supertropical eigenvalue of $A^\nabla$, obtained in the opposite order.
\end {exa}

 In the following conjecture we formulate how the $k^{th}$ coefficient  of the characteristic polynomial of $A^\nabla$ is related to the $(n-k)^{th}$ coefficient  of the characteristic polynomial of $A$.

\begin{con} If $A=(a_{i,j})\in M_n(R)$ is a non-singular matrix, then $$det(A)f_{A^\nabla}(x)\models_{gs} x^nf_A(x^{-1}).$$ An equivalent formulation is $$det(A)b_k\models_{gs} a_{n-k}\ \forall k=0,...,n.$$\end{con}

We resolve this conjecture for some special cases described in Proposition 6.3 and Example 6.4.

\begin{pro} Let~$A\in M_n(R)$. We denote the characteristic polynomials of the matrix~$A=\left(a_{i,j}\right)$ and~$A^\nabla=\left(\frac{a'_{i,j}}{det(A)}\right)$ as~$f_A(x)=\sum_{k=0}^n a_kx^k$ and~$f_{A^\nabla}(x)=\sum_{k=0}^n b_kx^k$ respectively, noting that~$x^nf_A(x^{-1})=\sum_{k=0}^n a_{n-k}x^k.$ Then, $$det(A)b_k\models_{gs} a_{n-k},\ \text{where}\  k\in\{0,n-2,n-1,n\}.$$
\vskip 0.25 truecm
\end{pro}

\begin{proof}
\noindent (i) Since $det(A)=det(A^\nabla)^{-1}$, we have $$det(A)b_0=det(A)det(A^\nabla)=0=a_{n}$$ and $$det(A)b_n=det(A)=a_0.$$ 
\vskip 0.2 truecm

\noindent (ii) We have~$det(A)b_{n-1}=a_{1}$ since $$det(A)tr(A^\nabla)=det(A)b_{n-1}=det(A)\sum \frac{a'_{i,i}}{det(A)}=\sum a'_{i,i},$$ which  is indeed the sum of determinants of all~$(n-1)\times (n-1)$ principal sub-matrices in $A$.
\vskip 0.2 truecm

\noindent (iii) We have~$det(A)b_{n-2}\models_{gs} a_{2},$ where~$b_{n-2}$ is the sum of determinants of~$2\times 2$ principal sub matrices in~$A^\nabla$, and \begin{equation}\label{coeff}det(A)b_{n-2}=det(A)\sum_{^{^{1\leq i_{_1}<i_{_2}\leq n}}}\sum_{\sigma\in S_2}\frac{a'_{i_{_1},\sigma(i_{_1})}}{det(A)}\frac{a'_{i_{_2},\sigma(i_{_2})}}{det(A)}\ .\end{equation}
\vskip 0.2 truecm

For  a product of entries of $A=(a_{i,j})$, we abuse  the graph theory notations of $d_{out}(r)$ and~$d_{in}(r)$  to denote the number of appearances of $r$ as a left index and    as a right index, respectively.
 We use ~\cite[Proposition 3.17]{STMA},   stating that  a  product of $kn$ entries of $A$ \begin{equation}\label{prod}\prod_{j=1}^{kn} a_{s_j,t_j},\ \  \text{where}\  d_{out}(s_j)=d_{in}(t_j)=k,\ \forall s_j,t_j=1,...,n,\end{equation}    is a product of $k$ permutation tracks of $A$.
\vskip 0.2 truecm

The numerator in \ref{coeff} \begin{equation}\label{adjper} \prod_{j=1}^2 a'_{i_j,\sigma(i_j)}\end{equation} is the product of two entries of the adjoin of $A$, located in  positions~$(i_1,\sigma(i_1))$ and ~$(i_2,\sigma(i_2))$. That is, \ref{adjper} is on the permutation track 
$i_1\mapsto \sigma(i_1),\ i_2\mapsto \sigma(i_2)$ of the $2\times 2$ principle sub-matrix of $i_1$ and $i_2$ in~$adj(A)$. Thus,~$\prod_{_{j=1}}^{^2} a'_{i_j,\sigma(i_j)}a_{\sigma(i_j),i_j}$ is the product of two permutation tracks of~$A$:
\begin{equation}\label{detper} [a_{1,\pi_1(1)}\cdots a_{n,\pi_1(n)} ][   a_{1,\pi_2(1)}\cdots a_{n,\pi_2(n)}]=\prod_{j=1}^2\prod_{r=1}^n  a_{r,\pi_j(r)}, \end{equation} where $ a_{\sigma(i_j),i_j}$ appears in the track of $\pi_j\in S_n$, and $$d_{out}(r)=d_{in}(\pi_j(r))=2,\ \forall r=1,...,n,\ j=1,2.$$
\vskip 0.2 truecm

Product \ref{detper} includes
\begin{equation}\label{2per} \prod_{j=1}^2 a_{\sigma(i_j),i_j},\end{equation} 
 which is a permutation track in a~$2\times 2$ principle sub-matrix of $A$, and $$d_{out}(\sigma(i_j))=d_{in}( i_j)=1,\ \forall j=1,2.$$

\vskip 0.2 truecm

Analyzing \ \ref{adjper}, we would like to find that when multiplied by~$\frac{det(A)}{det(A)^2}$, it either describes \ a summand in~$a_2$, dominated \ by one, or appears in \ another permutation track in a~$2\times 2$ principle sub-matrix of~$adj(A)$. We show that by applying a switch (a change of permutation on the same set of indices) on~$\sigma$. That is, we multiply \ref{adjper} by~$a_{\sigma(i_1),i_2} a_{\sigma(i_2),i_1}$, instead of~$a_{\sigma(i_1),i_1} a_{\sigma(i_2),i_2}$, as we did in \ref{detper}:
\begin{equation}\label{switch}a'_{i_1,\sigma(i_1)}a_{\sigma(i_1),i_2} a'_{i_2,\sigma(i_2)}a_{\sigma(i_2),i_1}.\end{equation}
Since the index counting in \ref{switch} do not chance (from \ref{detper}), it also describes the product of two permutation tracks of $A$:
\begin{equation}\label{switchper} [a_{1,\pi'_1(1)}\cdots a_{n,\pi'_1(n)}][ a_{1,\pi'_2(1)}\cdots a_{n,\pi'_2(n)}]=\prod_{j=1}^2\prod_{r=1}^n  a_{r,\pi'_j(r)},\end{equation}  where $d_{out}(r)=d_{in}(\pi'_j(r))=2\ \text{and}\ \pi'_j\in S_n,\ \forall r=1,...,n,\  j=1,2.$
\vskip 0.3 truecm

We look for the location of $ a_{\sigma(i_2),i_1}$ and $ a_{\sigma(i_1),i_2}$ in \ref{switchper}. 
\vskip 0.4 truecm

\noindent \underline{The summands obtaining $a_2$}:

\noindent If both appear in the same permutation track, without loss of generality $\pi'_1$, then the track of  $\pi'_2$, which is equal to, or dominated by $det(A)$, appears in \ref{adjper}. That is, we can consider $\pi'_2$ to be the permutation  attaining $det(A)$, otherwise, it is dominated by the summand having the track attaining $det(A)$ instead of $\pi'_2$. 

\noindent Next, since $a_{\sigma(i_1),i_2} a_{\sigma(i_2),i_1}$ is a track of a permutation of  $S_2$ in $A$, included in the track of $\pi'_1$, we have that~$\frac{ a_{1,\pi'_1(1)}\cdots a_{n,\pi'_1(n)}}{ a_{_{\sigma(i_1),i_2}} a_{_{\sigma(i_2),i_1}}}$ is a track of a permutation of  $S_{n-2}$ in $A$:
$$\frac{det(A)}{det(A)^2}\frac{ det(A)a_{1,\pi'_1(1)}\cdots a_{n,\pi'_1(n)}}{ a_{_{\sigma(i_1),i_2}} a_{_{\sigma(i_2),i_1}}}=\frac{ a_{1,\pi'_1(1)}\cdots a_{n,\pi'_1(n)}}{ a_{\sigma(i_1),i_2} a_{\sigma(i_2),i_1}},$$ which is a summand in $a_2$.
\vskip 0.4 truecm

\noindent \underline{The summands obtaining a ghost}:

\noindent If $ a_{\sigma(i_2),i_1}$ and $ a_{\sigma(i_1),i_2}$ appears in different permutation tracks, without loss of generality~$a_{\sigma(i_2),i_1}$ on the track of $\pi'_1$ and   $a_{\sigma(i_1),i_2}$ on the track of $\pi'_2$, then 
\begin{equation}\label{secper}\frac{[ a_{1,\pi'_1(1)}\cdots a_{n,\pi'_1(n)} ][   a_{1,\pi'_2(1)}\cdots a_{n,\pi'_2(n)}]}{ a_{\sigma(i_2),i_1} \ \ a_{\sigma(i_1),i_2}},\end{equation}  which is identical to \ref{adjper} (we  multiplied by   $a_{\sigma(i_2),i_1} a_{\sigma(i_1),i_2}$ instead of $ a_{\sigma(i_1),i_1} a_{\sigma(i_2),i_2}$ in the numerator, and then canceled it in the dominator), describes the product of two adjoint-entries, located in  positions~$(i_1,\sigma(i_2))$ and~$(i_2,\sigma(i_1))$, and described by the permutations~$\pi'_1$ and $\pi'_2$ respectively.

As a result, products \ref{adjper}  and \ref{secper} describe two identical summands   obtained on the permutation tracks of $\begin{cases}i_1\mapsto \sigma(i_1)\\ i_2\mapsto \sigma(i_2)\end{cases}$ and $\begin{cases}i_1\mapsto \sigma(i_2)\\ i_2\mapsto \sigma(i_1)\end{cases}$ in $adj(A)$, respectively.

\end{proof}

\begin{exa}

$ $

\noindent 1. If  $A=(a_{i,j})$ is a triangular, non-singular matrix, then $A^\nabla$ and all principal sub-matrices of $A$ and $A^\nabla$ are triangular. Denoting  $A^\nabla=(b_{i,j})$, we have $b_{i,i}=a_{i,i}^{-1}$.  Noting that the determinant of a triangular matrix comes from its diagonal, we get $$det(A)f_{A^\nabla}(x)=x^nf_A(x^{-1}).$$
\vskip 0.25 truecm

\noindent 2. By Proposition 6.3,  if  $A\in M_i(R)$, where $i=2,3$, then $$det(A)f_{A^\nabla}(x)\models x^if_A(x^{-1})\ \text{(with equality when}\ i=2),$$ since the coefficient of $x^{i-1}$ is the trace, the coefficient of $x^{i-2}$ is $b_{n-2}$  and the constant is the determinant. The $4\times 4$ case has been verified by a computer calculation.
\vskip 0.25 truecm

\end {exa}

Consequences of Conjecture 6.2 would be:
\vskip 0.2 truecm

\noindent 1. The inverse of every supertropical eigenvalue of $A$ is a supertropical eigenvalue of~$A^\nabla$, since if $\lambda$ is an eigenvalue of $A$, then $det(A)f_{A^\nabla}(\lambda^{-1})\models_{gs} (\lambda^{-1})^nf_A(\lambda)\in G.$
\vskip 0.2 truecm

\noindent 2.  If $f_{A^\nabla}\in T[x]$ then $det(A)f_{A^\nabla}(x)= x^nf_A(x^{-1})$ (follows immediately).
\vskip 0.2 truecm

\noindent 3. If $f_{A^{\nabla\nabla}}\in T[x]$ then $f_{A^{\nabla\nabla}}=f_A$, since applying (6.2) to $A$ and $A^\nabla$, we have $$det(A)^{-1}f_{A^{\nabla\nabla}}(x)\models_{gs}x^nf_{A^\nabla}(x^{-1})\ \text{ and }\ det(A)f_{A^{\nabla}}(x)\models_{gs}x^nf_{A}(x^{-1}).$$ Combining these two equations we get $$det(A)det(A)^{-1}f_{A^{\nabla\nabla}}(x)\models_{gs}det(A)x^nf_{A^\nabla}(x^{-1})\models_{gs}x^nx^{-n}f_{A}(x).$$ Then, if $f_{A^{\nabla\nabla}}\in T[x]$, we have that $f_{A^{\nabla\nabla}}=f_A$.


\vskip 0.95 truecm

\begin{thebibliography}{xx}
 \newcommand{\au}{\sc}
 \newcommand{\ti}{\it}


\bibitem{MPA}
{\au M.~Akian, R.~Bapat, and S.~Gaubert},
{\ti Max-plus algebra}.
Hogben L., Brualdi R., Greenbaum A., Mathias R. (eds.), Handbook of Linear Algebra. Chapman and Hall, London, 2006.
\vskip 0.1 truecm

\bibitem{LDTS}
{\au M. ~Akian, S. ~Gaubert and A. ~Guterman},
{\ti Linear dependence over tropical semirings and beyond}.
Tropical and Idempotent Mathematics, volume 495 of Contemporary Mathematics, pp. 1--38, AMS, 2009.
\vskip 0.1 truecm

\bibitem{TLI}
{\au A.~Buchholz},
{\ti Tropicalization of linear isomorphisms on plane elliptic curves}.
PhD dissertation, Mathematical Institute Georg-August, University Gottingen, April 2010.
\vskip 0.1 truecm

\bibitem{MA}
{\au P.~Butkovic},
{\ti Max-algebra: the linear algebra of combinatorics?}.
Linear Algebra Appl.~367 (2003), 313-335.
\vskip 0.1 truecm

\bibitem{CCP}
{\au P.~Butkovic},
{\ti On the coefficients of the max-algebraic characteristic polynomial and equation}.
In proceedings of the workshop on Max-algebra, Symposium of the International Federation of Automatic Control, Prague, 2001.
\vskip 0.1 truecm

\bibitem{ETCP}
{\au P.~Butkovic and L.~Murfitt},
{\ti Calculating essential terms of a characteristic maxpolynomial}.
CEJOR 8 (2000) 237-246.
\vskip 0.1 truecm

\bibitem{MNMX}
{\au R.A.~Cunininghame-Green},
{\ti Minimax algebra}.
Lecture notes in Economics and Mathematical Systems, no.~166, Springer,1979.
\vskip 0.1 truecm

\bibitem{LOP}
{\au M.~Fiedler, J.~Nedoma, J.~Ramik, J.~Rohn and K.~Zimmermann},
{\ti Linear optimization problems with inexact data}.
Springer, New York, 2006.
\vskip 0.1 truecm

\bibitem{MAA}
{\au S.~Gaubert and  M.~Plus},
{\ti Mathods and applications of $(max,+)$ linear algebra, in Lecture Notes in Computer Science}.
1200, Springer Verlag, New York, 1997.
\vskip 0.1 truecm

\bibitem{G.Ths}
{\au S.~Gaubert},
{\ti Th\'{e}orie des syst\`{e}mes lin\'{e}aires dans les dio\"{i}des}.
PhD dissertation, School of Mines. Paris, July  1992.
\vskip 0.1 truecm

\bibitem{S&A}
{\au J.S.~Golan},
{\ti Semirings and their applications}.
Kluwer Acad. Publ., Dordrecht, 1999.
\vskip 0.1 truecm

\bibitem{PA}
{\au M.~Gondran},
{\ti Path algebra and algorithms}.
In B. Roy, editor, Combinatorial programing: methods and applications, Reidel, Dordrecht, pp. 137-148, 1975.
\vskip 0.1 truecm

\bibitem{MPW}
{\au B.~Heidergott, G.J~Olsder and J.~van der Woude},
{\ti Max Plus at Work: Modeling and Analysis of Synchronized Systems: A Course on Max- Plus Algebra and Its Applications}.
 Princeton Univ. Press, 2006.
\vskip 0.1 truecm

\bibitem{TAG}
{\au I.~Itenberg, G.~Mikhalkin and E.~Shustin},
{\ti Tropical Algebraic Geometry}.
Oberwolfach Seminars, Vol. 35, Birkhauser, Basel e.a., 2007.
\vskip 0.1 truecm

\bibitem{TA}
{\au Z.~Izhakian},
{\ti Tropical arithmetic and matrix algebra}.
Comm. in Algebra 37(4):1445-1468, 2009.
\vskip 0.1 truecm

\bibitem{STA}
{\au Z.~Izhakian and L.~Rowen},
{\ti Supertropical algebra}.
Adv. Math. 225: 2222-2286, 2010.
\vskip 0.1 truecm

\bibitem{STLA}
{\au Z.~Izhakian, M.~Knebusch and L.~Rowen},
{\ti Supertropical linear algebra}.
To appear in Pacific Journal of Mathematics, 2013.
\vskip 0.1 truecm

\bibitem{STMA}
{\au Z. \ Izhakian and L.\ Rowen},
{\ti Supertropical matrix algebra}.
Israel Journal of Mathematics 182(1):383--424, 2011.
\vskip 0.1 truecm

\bibitem{STMA2}
{\au Z.~Izhakian and L.~Rowen},
{\ti Supertropical matrix algebra II: solving tropical equations}.
Israel Journal of Mathematics, 186(1):69-97,2011.% (Preprint at arXiv: 0902.2159, 16pp.)
\vskip 0.1 truecm

\bibitem{STMA3}
{\au Z.~Izhakian and L.~Rowen},
{\ti Supertropical matrix algebra III: Powers of matrices and their supertropical eigenvalues}.
Journal of Algebra, 341(1):125--149, 2011.% (Preprint at arXiv: 1008.0023.)
\vskip 0.1 truecm

\bibitem{IM}
{\au G. L.~Litvinov and V. P.~Maslov},
{\ti Idempotent mathematics: correspondence principle and applications}.
Russian Mathematical Surveys, 51, no. 6 (1996) 1210-1211.
\vskip 0.1 truecm

\bibitem{LDM}
{\au G. L.~Litvinov, A. Ya.~Rodionov, S.N.~Sergeev and A. V.~Sobolevski},
{\ti Universal algorithms for solving the matrix Bellmann equation over semirings}.
Preprint. Moscow, 2012. Submitted to Soft Computing. 
\vskip 0.1 truecm

\bibitem{T&I}
{\au G. L.~Litvinov and S.N.~Sergeev},
{\ti Tropical and Idempotent Mathenatics}.
Contemporary Mathematics, vol 495. AMS, Providence (2009).
\vskip 0.1 truecm

\bibitem{TGA}
{\au  G.~Mikhalkin},
{\ti Tropical geometry and its applications}.
In Proceedings of the ICM, Madrid, Spain, vol. II, 2006, pp. 827-852. arXiv: math/0601041v2 [math. AG]
\vskip 0.1 truecm

\bibitem{CP}
{\au  A.~Niv},
{\ti Characteristic Polynomials of Supertropical Matrices}.
Preprint at  arXiv:1201.0966v2, 2012. To appear in Communications in Algebra.
\vskip 0.1 truecm

\bibitem{FTM}
{\au  A.~Niv},
{\ti Factorization of tropical matrices}.
J. Algebra Appl., vol 13, 1350066 (2014).
\vskip 0.1 truecm

\bibitem{MCS}
{\au  C. ~Reutenauer and H. ~Straubing},
{\ti  Inversion of matrices over a  commutative semiring}.
J. Algebra, 88(2): 350-360, 1984.
\vskip 0.1 truecm


\bibitem{KS}
{\au  S.N.~Sergeev},
{\ti Multiorder, Kleene stars and cyclic projectors in the geometry of max cones}.
In G. L. Litvinov and S. N. Sergeev (Eds.), Contemporary mathematics: Vol. 495.
Tropical and idempotent mathematics (pp. 317–342). AMS, Providence (2009).

\vskip 0.1 truecm

\bibitem{visual}
{\au S.N.~Sergeev, H.~Schneider and P.~\ Butkovic},
{\ti On visualization scaling, subeigenvectors and Kleene stars in max algebra}.
Linear Algebra and its Applications, 431(12):2395--2406, 2009.
\vskip 0.1 truecm






\bibitem{CH}
{\au H.~Straubing},
{\ti A combinatorial proof of the Cayley-Hamilton Theorem}.
Discrete Math. 43 (2-3): 273-279, 1983.
\vskip 0.1 truecm

\bibitem{YL}
{\au M.~Yoeli}, 
{\ti A note on a generalization of boolean matrix theory}.
Amer.~Math.~Monthly 68 (1961), 552-557.



\end{thebibliography}
\end{document}